\documentclass[11pt,reqno,a4paper]{amsart}
\usepackage[T1]{fontenc}
\usepackage[utf8]{inputenc}
\usepackage{fbb}
\usepackage{amsmath,amssymb,amsthm,mathtools,microtype}
\usepackage{enumitem}
\usepackage[top=2.5cm,bottom=2.4cm,left=2.6cm,right=2.6cm]{geometry}
\usepackage{xcolor}
\usepackage{xr-hyper}
\usepackage[colorlinks=true,linkcolor=blue!45!green,citecolor=blue!45!green,urlcolor=blue!45!green]{hyperref}
\hypersetup{pdftitle={Gagliardo--Nirenberg extremals on manifolds with small holes},pdfauthor={Mayukh Mukherjee and Utsab Sarkar}}
\allowdisplaybreaks
\numberwithin{equation}{section}
\theoremstyle{plain}
\newtheorem{theorem}{Theorem}[section]
\newtheorem{introtheorem}{Theorem}

\newtheorem{lemma}[theorem]{Lemma}
\newtheorem{proposition}[theorem]{Proposition}
\newtheorem{corollary}[theorem]{Corollary}
\theoremstyle{definition}

\theoremstyle{remark}
\newtheorem{remark}[theorem]{Remark}

\title[Gagliardo--Nirenberg extremals on manifolds with small holes]{Gagliardo--Nirenberg extremals on manifolds with small holes}
\author{Mayukh Mukherjee}
\address{Department of Mathematics, Indian Institute of Technology Bombay, Powai, Mumbai 400076, India}
\email{mathmukherjee@gmail.com, mukherjee@math.iitb.ac.in}
\author{Utsab Sarkar}
\address{Department of Mathematics, Indian Institute of Technology Bombay, Powai, Mumbai 400076, India}
\email{utsab@math.iitb.ac.in, reachutsab@gmail.com}
\date{}

\begin{document}
\begin{abstract}
We study the sharp homogeneous Dirichlet Gagliardo--Nirenberg quotient on a closed Riemannian manifold with holes of small capacity and connected complements. For each fixed exponent $2<q<2^*$, the nonnegative $L^2$-normalized extremal exists and is unique when the first Dirichlet eigenvalue $\lambda_D$ is sufficiently small. If the ambient volume is $V$ and $\beta=n(q-2)/2$, the sharp constant satisfies $C_D=V^{1-q/2}\lambda_D^{-\beta/2}(1+O(\lambda_D))$. For the normalized first eigenfunction $\varphi_D$, we obtain $\|u_D-\varphi_D\|_{H^1}=O(\lambda_D^{3/2})$ and $C_D/W_{M_D}(\varphi_D)=1+O(\lambda_D^2)$, together with a quantitative stability estimate in terms of the relative deficit. The constants depend only on the ambient manifold and the exponent. On general compact manifolds with Dirichlet boundary conditions, we prove compactness above the Euclidean sharp constant and describe concentration at equality. Positive scalar curvature at one interior point implies attainment. For the $L^2$-critical nonlinear Schr\"odinger energy, the mass threshold for negative energy tends to zero as the capacity of the holes tends to zero, whereas the threshold for unboundedness below is unchanged.
\end{abstract}
\maketitle
\section{Introduction}
\label{gn:sec:introduction}

The sharp homogeneous Dirichlet Gagliardo--Nirenberg constant on a bounded Euclidean domain equals its whole-space value and is not attained. On a closed manifold with a smooth hole of sufficiently small capacity and connected complement, we prove that the supremum is attained by a unique nonnegative $L^2$-normalized extremal. We determine the asymptotics of the sharp constant, estimate the difference between the extremal and the first Dirichlet eigenfunction, and prove a stability inequality (Theorem~\ref{gn:thm:small-holes}).

Let $(M^n,g)$ be a smooth compact connected Riemannian manifold with smooth nonempty boundary. Throughout the paper,
\begin{equation}\label{gn:eq:range}
 n\ge2,\qquad 2<q<2^*,\qquad
 2^*=\begin{cases}2n/(n-2),&n\ge3,\\ \infty,&n=2,\end{cases}
\end{equation}
and
\begin{equation}\label{gn:eq:exponents}
 \beta=\frac{n(q-2)}2,\qquad \alpha=q-\beta>0.
\end{equation}
For $u\in H^1_0(M)\setminus\{0\}$, define
\begin{equation}\label{gn:eq:quotient}
 W_M(u)=\frac{\|u\|_{L^q(M)}^q}{\|u\|_{L^2(M)}^\alpha\|\nabla u\|_{L^2(M)}^\beta},
 \qquad C_M=\sup_{u\in H^1_0(M)\setminus\{0\}}W_M(u).
\end{equation}
We write $C_*=C_{\mathbb R^n}$ for the Euclidean sharp constant. Functions are real-valued unless stated otherwise. Norms and integrals on a manifold use its Riemannian volume unless another measure is displayed.

\subsection{Extremals for small holes}

Fix a smooth closed connected manifold $(N^n,g)$ of volume $V$. A hole $D\subset N$ is a nonempty smooth open set with proper closure, and $M_D=N\setminus\overline D$ carries the Dirichlet condition on $\partial D$. Put $C_D=C_{M_D}$ and $\lambda_D=\lambda_1(M_D)$. When $M_D$ is connected, let $\varphi_D>0$ be its first Dirichlet eigenfunction, normalized by $\|\varphi_D\|_2=1$. A tilde denotes zero extension to $N$. We use the mean-zero capacity
\begin{equation}\label{gn:eq:capacity-definition-intro}
 \operatorname{cap}_N(D)=
 \inf\left\{\int_N|\nabla v|^2:\ v\in H^1(N),\ v=1\ \hbox{a.e. on }D,\ \int_Nv=0\right\}.
\end{equation}
The mean-zero condition excludes nonzero constant functions.

For $f\in H^1_0(M_D)$ with $\|f\|_2=1$, define the relative deficit by
\begin{equation}\label{gn:eq:relative-deficit}
 \delta_D(f)=1-\frac{W_{M_D}(f)}{C_D}.
\end{equation}

\begin{introtheorem}[Uniqueness, stability, and spectral asymptotics]\label{gn:thm:small-holes}
There are $\lambda_0>0$, $\delta_0\in(0,1/2)$, and $K<\infty$, depending only on $N$ and $q$, such that the following holds for every smooth hole $D$ with connected complement and $0<\lambda_D<\lambda_0$.

One has $C_D>C_*$, and there is a unique nonnegative maximizer $u_D$ with $\|u_D\|_2=1$. As $\lambda_D\to0$,
\begin{equation}\label{gn:eq:spectral-equivalent}
 C_D=V^{1-q/2}\lambda_D^{-\beta/2}\bigl(1+O(\lambda_D)\bigr),
\end{equation}
and
\begin{align}
 \|\widetilde u_D-\widetilde\varphi_D\|_{H^1(N)}
 &=O(\lambda_D^{3/2}),\label{gn:eq:eigenfunction-approximation}\\
 \frac{C_D}{W_{M_D}(\varphi_D)}
 &=1+O(\lambda_D^2),\label{gn:eq:eigenfunction-quotient}\\
 \|\nabla u_D\|_2^2
 &=\lambda_D+O(\lambda_D^3),\label{gn:eq:energy-asymptotic}\\
 \|\widetilde u_D-V^{-1/2}\|_{H^1(N)}
 &=O(\lambda_D^{1/2}).\label{gn:eq:constant-approximation}
\end{align}
Every nonnegative $f\in H^1_0(M_D)$ with $\|f\|_2=1$ and $\delta_D(f)\le\delta_0$ satisfies
\begin{equation}\label{gn:eq:deficit-stability}
 \|\widetilde f-\widetilde u_D\|_{H^1(N)}^2
 \le K\lambda_D\,\delta_D(f).
\end{equation}
For a family of such holes, $\lambda_D\to0$ if and only if $\operatorname{cap}_N(D)\to0$, and
\begin{equation}\label{gn:eq:capacity-equivalent}
 C_D=V^{1-\alpha/2}\operatorname{cap}_N(D)^{-\beta/2}
 \bigl(1+O(\operatorname{cap}_N(D))\bigr).
\end{equation}
All error bounds are uniform over the holes.
\end{introtheorem}

The constants in Theorem~\ref{gn:thm:small-holes} are independent of curvature and collar bounds for $\partial D$.

\subsection{Compactness at the Euclidean threshold}

Let $B_{x,\ell,R}$ be the cutoff rescaling defined in \eqref{gn:eq:cutoff-bubble} from a Euclidean optimizer satisfying $\|Q_0\|_2=\|\nabla Q_0\|_2=1$.

\begin{introtheorem}[Threshold compactness]\label{gn:thm:threshold}
Under \eqref{gn:eq:range}, one has $C_*\le C_M<\infty$. If $C_M>C_*$, every maximizing sequence normalized by $\|u_k\|_2=1$ is precompact in $H^1_0(M)$, and $C_M$ is attained by a function positive in $M^\circ$.

If $C_M=C_*$, every such maximizing sequence has a subsequence that either converges strongly in $H^1_0(M)$ to an extremal, or satisfies
\begin{equation}\label{gn:eq:threshold-length}
 \ell_k=\|\nabla u_k\|_2^{-1}\longrightarrow0.
\end{equation}
In the latter case there are $x_k\in M^\circ$, $R_k\to\infty$, and signs $s_k\in\{-1,1\}$ such that $\ell_kR_k\to0$, the functions $B_{x_k,\ell_k,R_k}$ are supported in $M^\circ$, and
\begin{equation}\label{gn:eq:single-bubble}
 \|u_k-s_kB_{x_k,\ell_k,R_k}\|_2
 +\ell_k\|\nabla(u_k-s_kB_{x_k,\ell_k,R_k})\|_2\longrightarrow0.
\end{equation}
Their centers satisfy
\begin{equation}\label{gn:eq:boundary-repulsion}
 \frac{d_g(x_k,\partial M)}{\ell_k}\longrightarrow\infty.
\end{equation}
For every smooth bounded domain $\Omega\subset\mathbb R^n$, one has $C_\Omega=C_*$, and the supremum is not attained.
\end{introtheorem}

In the noncompact alternative, the centers may approach $\partial M$, but their distance from the boundary divided by $\ell_k$ tends to infinity. The limiting profile is therefore defined on $\mathbb R^n$.

\subsection{Scalar curvature and attainment}

Let $Q\in H^1(\mathbb R^n)$ be the positive radial solution, centered at the origin, of
\begin{equation}\label{gn:eq:ground-state}
 -\Delta Q+Q=Q^{q-1}\quad\hbox{in }\mathbb R^n.
\end{equation}
Exponential decay follows from the radial differential equation as in \cite[p.~155]{Strauss1977}; interior elliptic estimates give the same decay for the gradient. Thus $Q(y)+|\nabla Q(y)|\le Ce^{-c|y|}$ for some $C,c>0$. The test functions $U_{x,\varepsilon}$ are defined in \eqref{gn:eq:curvature-bubble}.

\begin{introtheorem}[Positive scalar curvature]\label{gn:thm:curvature}
If $\operatorname{Scal}_g(x)>0$ at one interior point, then $C_M>C_*$ and the supremum is attained. More precisely, uniformly for $x\in K\Subset M^\circ$,
\begin{equation}\label{gn:eq:curvature-expansion}
 W_M(U_{x,\varepsilon})
 =C_*\left(1+A_{n,q}\operatorname{Scal}_g(x)\varepsilon^2+O_K(\varepsilon^3)\right),
\end{equation}
where
\begin{equation}\label{gn:eq:positive-coefficient}
 A_{n,q}=\frac{\alpha}{12}
 +\frac{q-2}{12n}\frac{\displaystyle\int_{\mathbb R^n}|y|^2Q^q\,dy}
 {\displaystyle\int_{\mathbb R^n}Q^q\,dy}>0.
\end{equation}
Consequently, $C_M=C_*$ implies $\operatorname{Scal}_g\le0$ in $M^\circ$.
\end{introtheorem}

Positive scalar curvature is sufficient but not necessary for attainment. Theorem~\ref{gn:thm:small-holes} applies, for example, to flat tori with small geodesic balls removed.

\subsection{Constrained nonlinear Schr\"odinger energy}

At $q=2+4/n$, define
\begin{equation}\label{gn:eq:NLS-energy}
 \mathcal E_M(u)=\frac12\|\nabla u\|_2^2-\frac1q\|u\|_q^q,
 \qquad I_M(m)=\inf_{\substack{u\in H^1_0(M)\\\|u\|_2^2=m}}\mathcal E_M(u),
\end{equation}
and
\begin{equation}\label{gn:eq:two-masses}
 m_0(M)=\left(\frac{q}{2C_M}\right)^{n/2},
 \qquad m_* =\left(\frac{q}{2C_*}\right)^{n/2}.
\end{equation}
The thresholds for negative energy and for unboundedness below need not coincide; see \cite{AST2017} for metric graphs. The following theorem gives these thresholds on compact manifolds with Dirichlet boundary conditions and their asymptotics as the capacity of the holes tends to zero.

\begin{introtheorem}[Constrained energy minimization]\label{gn:thm:NLS}
There is a function of mass $m$ with negative energy if and only if $m>m_0(M)$. For $0<m<m_*$, the fixed-mass energy is coercive and $I_M(m)$ is attained. For $m>m_*$, one has $I_M(m)=-\infty$.

If $C_M>C_*$, every mass in $(m_0(M),m_*)$ admits a negative-energy constrained minimizer. For holes in the setting of Theorem~\ref{gn:thm:small-holes}, as $\lambda_D\to0$,
\begin{equation}\label{gn:eq:small-hole-mass}
 m_0(M_D)=\left(\frac q2\right)^{n/2}V\lambda_D^{n/2}
 \bigl(1+O(\lambda_D)\bigr)\longrightarrow0,
\end{equation}
whereas $m_*$ depends only on the dimension. In particular, $m_0(M_D)/m_*\to0$, and every fixed mass $m\in(0,m_*)$ eventually admits a negative-energy constrained minimizer.
\end{introtheorem}

\subsection{Related work and motivation}

The interpolation inequalities originate in the work of Gagliardo \cite{Gagliardo1958} and Nirenberg \cite{Nirenberg1959}. Weinstein \cite{Weinstein1983} characterized the Euclidean sharp constant by a variational problem, and Kwong \cite{Kwong1989} proved uniqueness of the positive radial solution of \eqref{gn:eq:ground-state}. Non-attainment on bounded Euclidean domains is proved in \cite[\S4.3]{CMMT2014}, which also observed that removing a small ball from a closed manifold can make the sharp constant exceed the Euclidean value and asked whether the supremum is attained on compact Riemannian manifolds with boundary. Banica \cite[Remark~6.1, p.~1668]{BanicaCPDE07} asked whether the sharp constant on hyperbolic space equals the Euclidean constant. Mukherjee \cite[Theorem~1.3]{Mukherjee2017} proved equality and non-attainment on hyperbolic space; Remark~2.1 of that paper records divergence of the supremum for punctured spheres in dimensions $n\ge3$. In this paper, we address the compact Riemannian manifold question. Compactness of the underlying manifold does not prevent maximizing sequences from concentrating at arbitrarily small scales. To exclude such concentration, we prove that the sharp constant strictly exceeds its Euclidean value for complements of holes of sufficiently small capacity. When the complement is connected, we also establish uniqueness and quantitative stability of the nonnegative normalized extremal.

Del Pino and Dolbeault \cite{DelPinoDolbeault02} determined sharp constants and optimizers for a family of Gagliardo--Nirenberg inequalities and related them to nonlinear diffusion. Cordero-Erausquin, Nazaret, and Villani \cite{CNV04} gave mass-transportation proofs of sharp Sobolev and Gagliardo--Nirenberg inequalities. Bianchi and Egnell \cite{BianchiEgnell91} proved quantitative stability for the critical Sobolev inequality. Carlen and Figalli \cite{CarlenFigalli13} prove stability for a planar Gagliardo--Nirenberg--Sobolev inequality involving the $L^4$ and $L^6$ norms and use fast diffusion to obtain stability for the logarithmic Hardy--Littlewood--Sobolev inequality and quantitative convergence for the critical-mass Keller--Segel equation. Bonforte, Dolbeault, Nazaret, and Simonov \cite{BDNSMemoir} use fast diffusion and entropy estimates to prove constructive stability for the $L^{p+1}$--$L^{2p}$ family and improved convergence rates toward self-similar solutions under suitable tail assumptions.

Riemannian inequalities with lower-order terms are treated in \cite{Brouttelande2003,CecconDuran2016,CecconDuranGodoi2026}; the last paper relates the second sharp constant to scalar curvature through weighted integrals of Euclidean optimizers. Cheng \cite{Cheng2025} obtains curvature restrictions for interpolation families with exponent pairs $2a$ and $a+1$. Theorem~\ref{gn:thm:curvature} concerns the homogeneous Dirichlet quotient \eqref{gn:eq:quotient}, which has an $L^2$ factor, throughout the range \eqref{gn:eq:range}. The relation between eigenvalue perturbations and capacity is studied in \cite{Courtois1995,AFHL2019}.

The concentration-compactness method is due to Lions \cite{Lions1984I,Lions1984II}. Coordinate estimates for manifolds with boundary are recorded in \cite{EscobarI}.

\section{Localization and Euclidean comparison}
\label{gn:sec:localization}

\subsection{Euclidean inequalities}

The Euclidean sharp inequality is
\begin{equation}\label{gn:eq:Euclidean-GN}
 \int_{\mathbb R^n}|f|^q\,dy
 \le C_*\left(\int_{\mathbb R^n}|f|^2\,dy\right)^{\alpha/2}
 \left(\int_{\mathbb R^n}|\nabla f|^2\,dy\right)^{\beta/2}.
\end{equation}
The nonzero real equality cases are $aQ(b(\,\cdot-y_0))$, with $a\ne0$ and $b>0$. To recall the classification, existence follows from \cite{Weinstein1983}. If $f$ is an optimizer, then so is $v=|f|$, and the Euler equation is
\[
 -\Delta v+A v=Bv^{q-1},\qquad
 A=\frac{\alpha\|\nabla v\|_2^2}{\beta\|v\|_2^2}>0,
 \quad B=\frac{q\|\nabla v\|_2^2}{\beta\|v\|_q^q}>0.
\]
Elliptic regularity and the maximum principle give $v>0$; local elliptic estimates and $v\in H^1(\mathbb R^n)$ give $v(y)\to0$ as $|y|\to\infty$. The symmetry theorem recalled in \cite[pp.~243--244]{Kwong1989} makes $v$ radial about some point. After rescaling to \eqref{gn:eq:ground-state}, Kwong's uniqueness theorem identifies it with $Q$. Since $v>0$, the function $f/v\in H^1_{\mathrm{loc}}(\mathbb R^n)$ takes only the values $\pm1$ and is constant. Conversely, multiplication, translation and dilation preserve equality. In particular, every real optimizer has constant sign and is nonzero everywhere.

The dilation $f(y)\mapsto\ell^{-n/2}f(y/\ell)$ preserves the $L^2$ norm, multiplies the squared gradient norm by $\ell^{-2}$ and the nonlinear integral by $\ell^{-\beta}$, and leaves the quotient unchanged.

Let $Q_0$ be the centered positive Euclidean optimizer normalized by
\begin{equation}\label{gn:eq:normalized-optimizer}
 \|Q_0\|_2=\|\nabla Q_0\|_2=1.
\end{equation}
It is obtained from $Q$ by multiplication by a constant and a dilation. Fix a smooth radial cutoff $\chi$, equal to one on $B_1$, zero outside $B_2$, and satisfying $0\le\chi\le1$.

\begin{lemma}[Dirichlet half-space and domains]\label{gn:lem:flat-domains}
The Dirichlet sharp constant on $\mathbb R^n_+$ equals $C_*$ and is not attained. The same conclusions hold on every smooth bounded Euclidean domain.
\end{lemma}

\begin{proof}
Zero extension gives the upper bound $C_*$. A compactly supported Euclidean near-optimizer can be translated and dilated into any interior ball, giving the reverse inequality. Attainment would make its zero extension a whole-space optimizer vanishing on an open set, contrary to the equality classification.
\end{proof}

\subsection{Local coordinates}

We use normal coordinates in the interior and Fermi coordinates at the boundary. The boundary normal points inward, and $\mathrm{II}(X,Y)=-g(\nabla_X\nu,Y)$, as in \cite[\S1.1]{EscobarI}. Put $K_g=\operatorname{tr}\mathrm{II}$. At a boundary center,
\begin{equation}\label{gn:eq:Fermi-first-order}
 \sqrt{|g(y)|}=1-K_g y_n+O(|y|^2),\qquad
 g^{ab}(y)=\delta^{ab}+2\mathrm{II}^{ab}y_n+O(|y|^2),
\end{equation}
with $g^{an}=0$ and $g^{nn}=1$; see \cite[Lemma~3.7]{EscobarI}. The coordinate bounds are uniform on fixed smaller neighborhoods by \cite[Lemmas~3.2 and~3.6]{EscobarI}. The same geometric construction applies in dimension two.

\begin{lemma}[Local comparison]\label{gn:lem:local-comparison}
Let $f\in H^1$ be supported in $B_{2R}$ in an interior chart, or in $B_{2R}^+$ with zero trace on the flat boundary in a Fermi chart. Write
\[
 P_{x,\varepsilon}f(z)=\varepsilon^{-n/2}f(\Phi_x^{-1}(z)/\varepsilon),
\]
extended by zero. For fixed $R$ and sufficiently small $\varepsilon$,
\begin{align}
 \|P_{x,\varepsilon}f\|_2^2
 &=\bigl(1+O_R(\varepsilon^\gamma)\bigr)\|f\|_2^2,\label{gn:eq:local-mass}\\
 \varepsilon^2\|\nabla P_{x,\varepsilon}f\|_2^2
 &=\bigl(1+O_R(\varepsilon^\gamma)\bigr)\|\nabla f\|_2^2,\label{gn:eq:local-energy}\\
 \varepsilon^\beta\|P_{x,\varepsilon}f\|_q^q
 &=\bigl(1+O_R(\varepsilon^\gamma)\bigr)\|f\|_q^q,\label{gn:eq:local-nonlinearity}
\end{align}
where $\gamma=2$ for an interior normal chart and $\gamma=1$ for a Fermi chart. Consequently,
\begin{equation}\label{gn:eq:local-quotient}
 W_M(P_{x,\varepsilon}f)=W_{\mathrm{flat}}(f)
 \bigl(1+O_R(\varepsilon^\gamma)\bigr).
\end{equation}
The estimates are uniform when the charts and their supports lie in the indicated coordinate neighborhoods.
\end{lemma}

\begin{proof}
Since $P_{x,\varepsilon}=\varepsilon^{-1}\mathcal T_{x,\varepsilon}$ in the notation of \cite{EscobarI}, its Dirichlet-energy identity~(3.11) gives \eqref{gn:eq:local-energy}. For the norms, use
\[
 \int_M|P_{x,\varepsilon}f|^r\,dV_g
 =\varepsilon^{n-nr/2}\int|f(y)|^r\sqrt{|g(\varepsilon y)|}\,dy.
\]
In normal charts both metric and density errors are $O_R(\varepsilon^2)$; in Fermi charts they are $O_R(\varepsilon)$. The powers of $\varepsilon$ cancel in the quotient.
\end{proof}

\begin{proposition}[Almost-sharp compact inequality]\label{gn:prop:almost-sharp}
For every $\eta>0$ there exists $B_\eta<\infty$ such that
\begin{equation}\label{gn:eq:almost-sharp}
 \|u\|_q^q\le(C_*+\eta)\|u\|_2^\alpha\|\nabla u\|_2^\beta
 +B_\eta\|u\|_2^q
 \qquad(u\in H^1_0(M)).
\end{equation}
On a fixed smooth closed manifold $N$, the same inequality holds for all $u\in H^1(N)$.
\end{proposition}

\begin{proof}
Given $\delta>0$, choose a finite normal/Fermi cover sufficiently fine that the local sharp inequality on each chart has constant at most $C_*+\delta$, by Lemmas~\ref{gn:lem:flat-domains} and~\ref{gn:lem:local-comparison}. Take real smooth functions $\theta_i$ subordinate to the cover with $\sum_i\theta_i^2=1$. Put $M_i=\|\theta_i u\|_2^2$ and $E_i=\|\nabla(\theta_i u)\|_2^2$. Minkowski's inequality in $L^{q/2}$ and then H\"older's inequality for the finite sum give
\begin{align*}
 \|u\|_q^2
 &=\left\|\sum_i|\theta_i u|^2\right\|_{q/2}
 \le\sum_i\|\theta_i u\|_q^2\\
 &\le(C_*+\delta)^{2/q}\sum_iM_i^{\alpha/q}E_i^{\beta/q}
 \le(C_*+\delta)^{2/q}\left(\sum_iM_i\right)^{\alpha/q}
 \left(\sum_iE_i\right)^{\beta/q}.
\end{align*}
Here $\alpha/q+\beta/q=1$. The identity $\sum_i\theta_i^2=1$ gives
\[
 \sum_iM_i=\|u\|_2^2,\qquad
 \sum_iE_i=\|\nabla u\|_2^2+\int_M\left(\sum_i|\nabla\theta_i|^2\right)u^2
 \le\|\nabla u\|_2^2+K_\delta\|u\|_2^2.
\]
Raising to the power $q/2$ yields
\begin{equation}\label{gn:eq:shifted-gradient-GN}
 \|u\|_q^q\le(C_*+\delta)\|u\|_2^\alpha
 \bigl(\|\nabla u\|_2^2+K_\delta\|u\|_2^2\bigr)^{\beta/2}.
\end{equation}
For $b>0$ and $\rho>0$, $(X+Y)^b\le(1+\rho)X^b+C_{b,\rho}Y^b$ for $X,Y\ge0$. Choose $\delta$ and $\rho$ with $(C_*+\delta)(1+\rho)\le C_*+\eta$ and apply this inequality with $b=\beta/2$. This proves \eqref{gn:eq:almost-sharp}. On a closed manifold the proof uses only interior charts.
\end{proof}

\begin{corollary}[Spectral bounds]\label{gn:cor:spectral-bounds}
Let $V_M=\operatorname{Vol}_g(M)$ and $\lambda_1(M)$ be the first Dirichlet eigenvalue. Then
\begin{equation}\label{gn:eq:spectral-bounds}
 V_M^{1-q/2}\lambda_1(M)^{-\beta/2}
 \le C_M\le C_*+\eta+B_\eta\lambda_1(M)^{-\beta/2}.
\end{equation}
In particular $C_M<\infty$.
\end{corollary}

\begin{proof}
For the $L^2$-normalized first eigenfunction $\varphi_1$, one has $\|\nabla\varphi_1\|_2^2=\lambda_1(M)$. H\"older's inequality gives $\|\varphi_1\|_q^q\ge V_M^{1-q/2}$, proving the lower bound. Divide \eqref{gn:eq:almost-sharp} by $\|u\|_2^\alpha\|\nabla u\|_2^\beta$ and use the Dirichlet Poincar\'e inequality for the upper bound.
\end{proof}

\begin{remark}[Dirichlet Cheeger constant]\label{gn:rem:cheeger}
Let
\[
 h_D(M)=\inf_{\substack{\varnothing\ne E\Subset M^\circ\\ E\ \mathrm{smooth}}}
 \frac{\operatorname{Per}_g(E)}{\operatorname{Vol}_g(E)}.
\]
Cheeger's inequality \cite{Cheeger1970} gives $\lambda_1(M)\ge h_D(M)^2/4$. Indeed, the coarea formula applied to $u^2$, for $u\in C_c^\infty(M^\circ)$, gives
\[
 h_D(M)\|u\|_2^2\le\int_M|\nabla(u^2)|\le2\|u\|_2\|\nabla u\|_2.
\]
Consequently, \eqref{gn:eq:spectral-bounds} implies
\begin{equation}\label{gn:eq:cheeger-upper}
 C_M\le C_*+\eta+2^\beta B_\eta h_D(M)^{-\beta}.
\end{equation}
If a family has fixed positive volume, uniform $B_\eta$ for one fixed $\eta>0$, and $\lambda_1(M_j)\asymp h_D(M_j)^\theta\to0$ for some $\theta>0$, then \eqref{gn:eq:spectral-bounds} gives $C_{M_j}\asymp h_D(M_j)^{-\theta\beta/2}$.
\end{remark}

\section{Compactness of maximizing sequences}
\label{gn:sec:compactness}

\subsection{Compactness above the Euclidean sharp constant}

\begin{proof}[Proof of the strict part of Theorem~\ref{gn:thm:threshold}]
Insert compactly supported Euclidean near-optimizers to obtain $C_M\ge C_*$. Let $\|u_k\|_2=1$ and $W_M(u_k)\to C_M$. If the gradients are unbounded, a subsequence has $\|\nabla u_k\|_2\to\infty$, and \eqref{gn:eq:almost-sharp} gives $\limsup_kW_M(u_k)\le C_*+\eta$ for every $\eta>0$. Thus $C_M>C_*$ forces boundedness in $H^1_0(M)$.

After extraction, $u_k\rightharpoonup u$ in $H^1_0(M)$ and $u_k\to u$ in $L^2(M)\cap L^q(M)$, with $\|u\|_2=1$. Lower semicontinuity gives $W_M(u)\ge C_M$, hence equality, and
\[
 \|\nabla u_k\|_2^\beta
 =\frac{\|u_k\|_q^q}{W_M(u_k)}
 \longrightarrow\frac{\|u\|_q^q}{C_M}=\|\nabla u\|_2^\beta.
\]
Convergence is therefore strong in $H^1_0(M)$. The argument applies to every subsequence. Taking the absolute value gives a nonnegative extremal.
\end{proof}

\begin{proposition}[Equation for an extremal]\label{gn:prop:extremal-equation}
Every nonnegative extremal $u$ satisfies
\begin{equation}\label{gn:eq:extremal-equation}
 -\Delta_g u+\frac{\alpha\|\nabla u\|_2^2}{\beta\|u\|_2^2}u
 =\frac{q\|\nabla u\|_2^2}{\beta\|u\|_q^q}u^{q-1}
 \quad\hbox{in }M,\qquad u|_{\partial M}=0.
\end{equation}
It is positive in $M^\circ$ and belongs to $C^2(\overline M)$.
\end{proposition}

\begin{proof}
Differentiate $\log W_M(u)$ against $\phi\in H^1_0(M)$ and multiply by $\|\nabla u\|_2^2/\beta$. This gives \eqref{gn:eq:extremal-equation} weakly. Subcritical elliptic regularity gives boundedness and then $C^2$ regularity up to the smooth boundary. Both coefficients in \eqref{gn:eq:extremal-equation} are positive, and the strong maximum principle gives interior positivity.
\end{proof}

\subsection{Rescaled maximizing sequences}

In an interior normal chart $\Psi_x$, put
\begin{equation}\label{gn:eq:cutoff-bubble}
 B_{x,\ell,R}(z)=\ell^{-n/2}\chi\!\left(\frac{\Psi_x^{-1}(z)}{\ell R}\right)
 Q_0\!\left(\frac{\Psi_x^{-1}(z)}{\ell}\right),
\end{equation}
extended by zero. This notation is used only when the support lies in $M^\circ$.

We next consider a sequence with
\begin{equation}\label{gn:eq:blowup-assumptions}
 \|u_k\|_2=1,\qquad \ell_k=\|\nabla u_k\|_2^{-1}\to0,
 \qquad W_M(u_k)\to C_*.
\end{equation}
Set $g_k=\ell_k^{-2}g$ and $v_k=\ell_k^{n/2}u_k$. Then
\begin{equation}\label{gn:eq:rescaled-normalization}
 \int_Mv_k^2\,dV_{g_k}=1,\qquad
 \int_M|\nabla v_k|_{g_k}^2\,dV_{g_k}=1,\qquad
 \int_M|v_k|^q\,dV_{g_k}\to C_*.
\end{equation}

\begin{lemma}[A nonzero rescaled profile]\label{gn:lem:nonzero-profile}
There are points $z_k\in\overline M$ and a subsequence such that one of the following holds.

If $d_g(z_k,\partial M)/\ell_k\to\infty$, the rescalings in normal coordinates at $z_k$ converge weakly in $H^1_{\mathrm{loc}}(\mathbb R^n)$ and strongly in $L^2_{\mathrm{loc}}\cap L^q_{\mathrm{loc}}$ to a nonzero $w\in H^1(\mathbb R^n)$.

Otherwise, after choosing nearest boundary points $a_k$ and Fermi coordinates there, the rescalings converge in the same local senses on $\mathbb R^n_+$ to a nonzero $w\in H^1_0(\mathbb R^n_+)$. In either case,
\begin{equation}\label{gn:eq:profile-fractions}
 0<m:=\|w\|_2^2\le1,\qquad 0<e:=\|\nabla w\|_2^2\le1.
\end{equation}
\end{lemma}

\begin{proof}
In the metrics $g_k$, the normal/Fermi atlas at scale $\ell_k$ gives coordinate neighborhoods of bounded radius and overlap, with uniformly bounded cutoff gradients. Apply the nonvanishing argument of \cite{Lions1984I,Lions1984II} on this cover.

Set $r_0=2+4/n$. The Euclidean GN inequality applied to a cutoff of $v_k$ in each neighborhood, followed by summation, gives
\begin{equation}\label{gn:eq:vanishing-estimate}
 \|v_k\|_{L^{r_0}(g_k)}^{r_0}
 \le C\left(\sup_{z\in\overline M}\int_{B_{g_k}(z,\rho)}v_k^2\,dV_{g_k}\right)^{2/n}
 \int_M\bigl(|\nabla v_k|_{g_k}^2+v_k^2\bigr)\,dV_{g_k},
\end{equation}
for fixed $\rho>0$. The exponent $r_0$ makes the gradient power two; cutoff derivatives give the $L^2$ term. Boundary cutoffs are extended by zero.

For every fixed $2<s<2^*$, put $\beta_s=n(s-2)/2$. Proposition~\ref{gn:prop:almost-sharp}, with exponent $s$, gives
\[
 \|v_k\|_{L^s(g_k)}^s
 =\ell_k^{\beta_s}\|u_k\|_s^s
 \le C_s+B_s\ell_k^{\beta_s}.
\]
If the supremum in \eqref{gn:eq:vanishing-estimate} tended to zero, the vanishing estimate and interpolation with $L^2$, or with an exponent $s>\max\{q,r_0\}$, would give $\|v_k\|_{L^q(g_k)}\to0$, contradicting \eqref{gn:eq:rescaled-normalization}. Hence there exist $z_k$ and $c>0$ such that
\[
 \int_{B_g(z_k,\rho\ell_k)}u_k^2\,dV_g\ge c.
\]
Pass to a subsequence for which $d_g(z_k,\partial M)/\ell_k$ either tends to infinity or remains bounded. In the first case interior normal charts have rescaled radii tending to infinity. In the second, Fermi charts at $a_k$ contain the indicated balls inside a fixed rescaled half-ball. The pulled-back metrics converge to the Euclidean metric on every fixed ball or half-ball.

The bounds \eqref{gn:eq:rescaled-normalization} give weak $H^1$ convergence and, by the subcritical Rellich theorem, strong local $L^2$ and $L^q$ convergence. The displayed lower bound makes the limit nonzero. Lower semicontinuity on expanding balls gives \eqref{gn:eq:profile-fractions}. In the boundary case local trace continuity gives zero trace, and the global $H^1$ bound puts $w$ in $H^1_0(\mathbb R^n_+)$. A nonzero $L^2$ function on either model cannot have zero gradient.
\end{proof}

\begin{lemma}[Norm decomposition]\label{gn:lem:one-profile-splitting}
Let $w,m,e$ be as in Lemma~\ref{gn:lem:nonzero-profile}. There are cutoffs of its pushforwards $w_k\in H^1_0(M)$, supported in coordinate neighborhoods whose radii in the original metric tend to zero, such that, for $r_k=u_k-w_k$,
\begin{align}
 \|r_k\|_2^2&\to1-m,\label{gn:eq:mass-splitting}\\
 \ell_k^2\|\nabla r_k\|_2^2&\to1-e,\label{gn:eq:energy-splitting}\\
 \ell_k^\beta\bigl(\|u_k\|_q^q-\|r_k\|_q^q\bigr)&\to\|w\|_q^q.\label{gn:eq:nonlinear-splitting}
\end{align}
\end{lemma}

\begin{proof}
At a fixed rescaled radius $R$, let $w_{k,R}=P_{z_k,\ell_k}(\chi_Rw)$ in the interior case and $w_{k,R}=P_{a_k,\ell_k}(\chi_Rw)$ in the boundary case, using the charts from Lemma~\ref{gn:lem:nonzero-profile}; here $\chi_R(y)=\chi(y/R)$. Weak local convergence and the local metric comparison give
\begin{align*}
 \langle u_k,w_{k,R}\rangle_2&\to\langle w,\chi_Rw\rangle_2,&
 \|w_{k,R}\|_2^2&\to\|\chi_Rw\|_2^2,\\
 \ell_k^2\langle\nabla u_k,\nabla w_{k,R}\rangle_2
 &\to\langle\nabla w,\nabla(\chi_Rw)\rangle_2,&
 \ell_k^2\|\nabla w_{k,R}\|_2^2&\to\|\nabla(\chi_Rw)\|_2^2.
\end{align*}
Strong local $L^q$ convergence also gives
\[
 \ell_k^\beta\bigl(\|u_k\|_q^q-\|u_k-w_{k,R}\|_q^q\bigr)
 \longrightarrow\int\bigl(|w|^q-|w-\chi_Rw|^q\bigr).
\]
This is the localized form of the nonlinear splitting of Br\'ezis and Lieb \cite{BrezisLieb1983}. Since $w\in H^1$, one has $\chi_Rw\to w$ in $H^1$ and $L^q$ as $R\to\infty$. In particular, the cutoff-gradient term is bounded by $C R^{-1}\|w\|_{L^2(R<|y|<2R)}$ and tends to zero.

Choose $R=R_k\to\infty$ sufficiently slowly that the preceding limits hold along the diagonal, $\ell_kR_k\to0$, and the supports remain in the coordinate domains. In the interior case require in addition $2\ell_kR_k<d_g(z_k,\partial M)$. Expanding the two squared norms proves \eqref{gn:eq:mass-splitting} and \eqref{gn:eq:energy-splitting}; the nonlinear limit gives \eqref{gn:eq:nonlinear-splitting}.
\end{proof}

\begin{lemma}[An elementary inequality]\label{gn:lem:strict-splitting}
Let $a,b>0$ with $s=a+b>1$. If $0\le m,e\le1$, then
\begin{equation}\label{gn:eq:strict-splitting}
 m^ae^b+(1-m)^a(1-e)^b\le1.
\end{equation}
If $m,e>0$, equality holds only when $m=e=1$.
\end{lemma}

\begin{proof}
Weighted Young's inequality bounds the left-hand side by
\[
 \frac as\bigl(m^s+(1-m)^s\bigr)
 +\frac bs\bigl(e^s+(1-e)^s\bigr)\le1.
\]
Since $t^s+(1-t)^s=1$ only at $t=0,1$, positive $m,e$ force $m=e=1$ in the equality case.
\end{proof}

\begin{proof}[Proof of the concentration part of Theorem~\ref{gn:thm:threshold}]
First assume \eqref{gn:eq:blowup-assumptions}. Apply the almost-sharp inequality to $r_k$ in Lemma~\ref{gn:lem:one-profile-splitting}, multiply by $\ell_k^\beta$, and let $k\to\infty$ and then $\eta\downarrow0$. Since $\|r_k\|_2$ is bounded, the lower-order term tends to zero. Thus
\[
 \limsup_k\ell_k^\beta\|r_k\|_q^q
 \le C_*(1-m)^{\alpha/2}(1-e)^{\beta/2}.
\]
The flat inequality, with zero extension if necessary, gives $\|w\|_q^q\le C_*m^{\alpha/2}e^{\beta/2}$. The nonlinear splitting now yields
\[
 C_*\le C_*\left(m^{\alpha/2}e^{\beta/2}
 +(1-m)^{\alpha/2}(1-e)^{\beta/2}\right)\le C_*.
\]
Since $(\alpha+\beta)/2=q/2>1$, Lemma~\ref{gn:lem:strict-splitting} gives $m=e=1$. Moreover $w$ attains the flat sharp constant. The half-space case is impossible by Lemma~\ref{gn:lem:flat-domains}. Hence $d_g(z_k,\partial M)/\ell_k\to\infty$.

The whole-space equality classification and $m=e=1$ give $w=sQ_0(\,\cdot-y_0)$ for some $s\in\{-1,1\}$ and $y_0\in\mathbb R^n$. Include the fixed translation in the choice of centers. Smooth dependence of the coordinate maps, applied to the fixed profile and followed by a diagonal choice of cutoff radii, changes the approximation by $o(1)$ in $L^2$ and by $o(\ell_k^{-1})$ in the gradient norm. Equations~\eqref{gn:eq:mass-splitting} and~\eqref{gn:eq:energy-splitting} give \eqref{gn:eq:single-bubble}, and the translated centers still satisfy \eqref{gn:eq:boundary-repulsion}.

For a maximizing sequence at $C_M=C_*$, a subsequence with bounded gradient norms converges strongly by the direct argument used in the strict case. Otherwise select a subsequence with gradients tending to infinity and apply the preceding proof. The bounded Euclidean-domain statement is Lemma~\ref{gn:lem:flat-domains}.
\end{proof}

\section{Scalar curvature and attainment}
\label{gn:sec:curvature}

\subsection{Moment identities and cutoffs}

For $x\in K\Subset M^\circ$ and sufficiently small $\varepsilon$, define
\begin{equation}\label{gn:eq:curvature-bubble}
 U_{x,\varepsilon}(z)=\varepsilon^{-n/2}
 \chi\!\left(\frac{\Psi_x^{-1}(z)}{\varepsilon R(\varepsilon)}\right)
 Q\!\left(\frac{\Psi_x^{-1}(z)}{\varepsilon}\right),
 \qquad R(\varepsilon)=L|\log\varepsilon|,
\end{equation}
extended by zero outside the normal chart, where $L$ is a sufficiently large fixed constant.

For $j=0,2$, write
\[
 M_j=\int_{\mathbb R^n}|y|^jQ^2,\qquad
 E_j=\int_{\mathbb R^n}|y|^j|\nabla Q|^2,\qquad
 N_j=\int_{\mathbb R^n}|y|^jQ^q.
\]
Exponential decay gives, for every fixed integer $j\ge0$,
\begin{equation}\label{gn:eq:exponential-tail}
 \int_{|y|>R}(1+|y|)^j\bigl(Q^2+|\nabla Q|^2+Q^q\bigr)\,dy
 \le C_j e^{-cR},\qquad R\ge1.
\end{equation}
The cutoff argument of \cite[Lemma~A.1]{EscobarI}, using \eqref{gn:eq:exponential-tail}, bounds all required moment errors by $O(e^{-c'R})$. Choose $L$ so these are $O(\varepsilon^4)$. The cutoff profiles have uniformly bounded third weighted moments.

\begin{lemma}[Identities for the Euclidean optimizer]\label{gn:lem:ground-state-identities}
One has
\begin{equation}\label{gn:eq:Pohozaev}
 M_0=\frac\alpha qN_0,\qquad E_0=\frac\beta qN_0,
\end{equation}
and
\begin{equation}\label{gn:eq:weighted-identity}
 E_2+M_2=N_2+nM_0.
\end{equation}
\end{lemma}

\begin{proof}
Testing \eqref{gn:eq:ground-state} by $Q$ gives $E_0+M_0=N_0$. For $Q_t(y)=t^{n/2}Q(ty)$, differentiating $\frac12\|\nabla Q_t\|_2^2+\frac12\|Q_t\|_2^2-\frac1q\|Q_t\|_q^q$ at $t=1$ gives $E_0=\beta N_0/q$. Testing the equation by $|y|^2Q$ and using $2\int Q\,y\cdot\nabla Q=-nM_0$ gives \eqref{gn:eq:weighted-identity}. Exponential decay justifies these operations.
\end{proof}

\subsection{The interior expansion}

\begin{proof}[Proof of Theorem~\ref{gn:thm:curvature}]
Let $f_\varepsilon=\chi_{R(\varepsilon)}Q$. In normal coordinates at $x\in K$, uniformly on the shrinking support,
\begin{equation}\label{gn:eq:normal-volume}
 \sqrt{|g(\varepsilon y)|}
 =1-\frac{\varepsilon^2}{6}\operatorname{Ric}_g(x)_{ij}y_i y_j
 +O_K(\varepsilon^3|y|^3).
\end{equation}
For radial $f_\varepsilon$, the Gauss lemma gives $g^{ij}(\varepsilon y)\partial_i f_\varepsilon\partial_j f_\varepsilon=|\partial_r f_\varepsilon|^2$. Radial integration contracts the quadratic volume term to $\operatorname{Scal}_g(x)/n$. Thus
\begin{align}
 \|U_{x,\varepsilon}\|_2^2
 &=M_0-\frac{\operatorname{Scal}_g(x)}{6n}\varepsilon^2M_2+O_K(\varepsilon^3),\label{gn:eq:curvature-mass}\\
 \varepsilon^2\|\nabla U_{x,\varepsilon}\|_2^2
 &=E_0-\frac{\operatorname{Scal}_g(x)}{6n}\varepsilon^2E_2+O_K(\varepsilon^3),\label{gn:eq:curvature-energy}\\
 \varepsilon^\beta\|U_{x,\varepsilon}\|_q^q
 &=N_0-\frac{\operatorname{Scal}_g(x)}{6n}\varepsilon^2N_2+O_K(\varepsilon^3).\label{gn:eq:curvature-nonlinearity}
\end{align}
The remainders are uniform because the third weighted moments are bounded independently of the cutoff radius. Since $N_0/(M_0^{\alpha/2}E_0^{\beta/2})=C_*$, expansion of the quotient yields
\begin{equation}\label{gn:eq:raw-curvature-coefficient}
 \frac{W_M(U_{x,\varepsilon})}{C_*}
 =1-\frac{\operatorname{Scal}_g(x)}{6n}\varepsilon^2
 \left(\frac{N_2}{N_0}-\frac\alpha2\frac{M_2}{M_0}
 -\frac\beta2\frac{E_2}{E_0}\right)+O_K(\varepsilon^3).
\end{equation}
By Lemma~\ref{gn:lem:ground-state-identities},
\begin{align*}
 \frac{N_2}{N_0}-\frac\alpha2\frac{M_2}{M_0}-\frac\beta2\frac{E_2}{E_0}
 &=\frac{N_2}{N_0}-\frac{q(M_2+E_2)}{2N_0}\\
 &=-\frac{q-2}{2}\frac{N_2}{N_0}-\frac{n\alpha}{2}.
\end{align*}
Substituting into \eqref{gn:eq:raw-curvature-coefficient} gives \eqref{gn:eq:curvature-expansion} and \eqref{gn:eq:positive-coefficient}. If scalar curvature is positive at an interior point, $U_{x,\varepsilon}$ has quotient larger than $C_*$ for sufficiently small $\varepsilon$. Theorem~\ref{gn:thm:threshold} gives attainment. The last assertion is the contrapositive.
\end{proof}

\begin{corollary}[Small balls]\label{gn:cor:small-balls}
For a fixed interior point $x$, $C_{B_g(x,r)}\to C_*$ as $r\downarrow0$. If $\operatorname{Scal}_g(x)>0$, then every sufficiently small such ball has $C_{B_g(x,r)}>C_*$ and admits an extremal.
\end{corollary}

\begin{proof}
Zero extension in a normal chart and the metric comparison give $C_{B_g(x,r)}\le C_*(1+Cr^2)$, while local insertion gives the reverse lower bound $C_*$. If scalar curvature is positive at $x$, choose $\varepsilon$ so that $U_{x,\varepsilon}$ is supported in $B_g(x,r)$ and use Theorem~\ref{gn:thm:curvature}.
\end{proof}

\section{Extremals on manifolds with small holes}
\label{gn:sec:holes}

Throughout this section $(N,g)$ is fixed, smooth, closed, and connected, and $V=\operatorname{Vol}(N)$. Put $\Lambda=\lambda_+(N)>0$, the first positive Laplace eigenvalue. Then
\begin{equation}\label{gn:eq:ambient-Poincare}
 \|f-\overline f\|_2^2\le\Lambda^{-1}\|\nabla f\|_2^2,
 \qquad \overline f=V^{-1}\int_N f.
\end{equation}
We identify Dirichlet functions with their zero extensions to $N$ and suppress the tildes in the proofs. The $H^1$ norm is $\|f\|_{H^1(N)}^2=\|f\|_2^2+\|\nabla f\|_2^2$. All constants in this section depend only on $N$ and the fixed exponent $q$.

\subsection{Small energy and localization near the first eigenspace}

\begin{lemma}[Nonlinear norm at small energy]\label{gn:lem:small-energy-norm}
There are $E_*>0$ and $K<\infty$ such that, if $f\in H^1(N)$, $f\ge0$, $\|f\|_2=1$, and $E=\|\nabla f\|_2^2\le E_*$, then
\begin{equation}\label{gn:eq:small-energy-norm}
 V^{1-q/2}\le\|f\|_q^q\le V^{1-q/2}+KE,
 \qquad \|f-V^{-1/2}\|_{H^1(N)}\le K E^{1/2}.
\end{equation}
\end{lemma}

\begin{proof}
Write $f=a+h$, with $a=\overline f\ge0$ and $\int_Nh=0$. The Poincar\'e inequality and the Sobolev embedding give $\|h\|_2^2\le E/\Lambda$ and $\|h\|_q\le K E^{1/2}$. Since $Va^2=1-\|h\|_2^2$, one has $a=V^{-1/2}+O(E)$. The inequality $|(a+t)^q-a^q-qa^{q-1}t|\le K(t^2+|t|^q)$, valid for these $a$ and $a+t\ge0$, yields
\[
 \int_N f^q=Va^q+O(\|h\|_2^2+\|h\|_q^q)=V^{1-q/2}+O(E).
\]
H\"older's inequality gives the lower bound; the same decomposition gives the $H^1$ estimate.
\end{proof}

Fix a smooth hole $D$ with connected complement. Write $\lambda=\lambda_D$, $\varphi=\varphi_D$, and $S=V^{1-q/2}$. Corollary~\ref{gn:cor:spectral-bounds} gives
\begin{equation}\label{gn:eq:uniform-spectral-lower}
 C_D\ge S\lambda^{-\beta/2}.
\end{equation}
For small $\lambda$ this exceeds $C_*$, so Theorem~\ref{gn:thm:threshold} supplies a nonnegative normalized maximizer.

\begin{lemma}[Energy of near-maximizers]\label{gn:lem:near-maximizer-energy}
For sufficiently small $\lambda$, every nonnegative $f\in H^1_0(M_D)$ with $\|f\|_2=1$ and $\delta=\delta_D(f)\le1/2$ satisfies
\begin{equation}\label{gn:eq:near-maximizer-energy}
 \lambda\le E:=\|\nabla f\|_2^2
 \le\lambda\bigl(1+K(\delta+\lambda)\bigr).
\end{equation}
Moreover,
\begin{equation}\label{gn:eq:leading-constant-proof}
 S\le C_D\lambda^{\beta/2}\le S+K\lambda.
\end{equation}
\end{lemma}

\begin{proof}
The inequality on $N$ \eqref{gn:eq:shifted-gradient-GN} gives
\[
 \tfrac12 S\lambda^{-\beta/2}\le W_{M_D}(f)
 \le K(1+E^{-\beta/2}),
\]
so $E\le K\lambda$ for small $\lambda$. The Rayleigh principle gives $E\ge\lambda$. By Lemma~\ref{gn:lem:small-energy-norm}, $\|f\|_q^q\le S+K\lambda$, whence
\[
 \frac E\lambda\le
 \left(\frac{1+K\lambda}{1-\delta}\right)^{2/\beta}
 \le1+K'(\delta+\lambda).
\]
This proves \eqref{gn:eq:near-maximizer-energy}. For a maximizer $u$, it also gives $\|u\|_q^q\le S+K\lambda$. Together with $\|\nabla u\|_2^2\ge\lambda$ and \eqref{gn:eq:uniform-spectral-lower}, this proves \eqref{gn:eq:leading-constant-proof}.
\end{proof}

The second Dirichlet eigenvalue satisfies
\begin{equation}\label{gn:eq:ambient-gap}
 \lambda_2(M_D)\ge\Lambda.
\end{equation}
Indeed, every two-dimensional subspace of zero-extended Dirichlet functions contains a nonzero function of mean zero on $N$, and the min--max principle applies. Set
\[
 X_D=\{z\in H^1_0(M_D):\langle z,\varphi\rangle_2=0\},\qquad
 B_D(z,h)=\int_N\nabla z\cdot\nabla h-\lambda\int_Nzh.
\]
For $\lambda\le\Lambda/2$, \eqref{gn:eq:ambient-gap} gives
\begin{equation}\label{gn:eq:complement-gap}
 c_N\|z\|_{H^1(N)}^2\le B_D(z,z)\le\|z\|_{H^1(N)}^2
 \qquad(z\in X_D),
\end{equation}
where $c_N>0$ is independent of $D$. If $f\ge0$ and $\|f\|_2=1$, write
\begin{equation}\label{gn:eq:sphere-coordinates}
 f=F_D(z):=a(z)\varphi+z,\qquad
 a(z)=\sqrt{1-\|z\|_2^2},\qquad z\in X_D.
\end{equation}
This representation uses $a(z)=\langle f,\varphi\rangle_2>0$. Orthogonality and the eigenfunction equation give the exact identity
\begin{equation}\label{gn:eq:exact-sphere-energy}
 \|\nabla F_D(z)\|_2^2=\lambda+B_D(z,z).
\end{equation}
\subsection{Strict concavity and quantitative stability}

For sufficiently small $\lambda$, define the convex neighborhood
\[
 \mathcal U_D=\{z\in X_D:B_D(z,z)<\lambda/2\}.
\]
By \eqref{gn:eq:complement-gap}, $a(z)$ is uniformly bounded away from zero on this set. We use \eqref{gn:eq:sphere-coordinates} for all real $z\in\mathcal U_D$, and put
\begin{equation}\label{gn:eq:log-quotient}
 J_D(z)=\int_N|F_D(z)|^q,\qquad
 \mathcal L_D(z)=\log J_D(z)-\frac\beta2\log\bigl(\lambda+B_D(z,z)\bigr).
\end{equation}
Thus $\mathcal L_D(z)=\log W_{M_D}(F_D(z))$.

\begin{lemma}[Uniform strict concavity]\label{gn:lem:strict-concavity}
For sufficiently small $\lambda$,
\begin{equation}\label{gn:eq:log-Hessian}
 D^2\mathcal L_D(z)[h,h]
 \le-\frac c\lambda\|h\|_{H^1(N)}^2
 \qquad(z\in\mathcal U_D,\ h\in X_D),
\end{equation}
where $c>0$ depends only on $N$ and $q$. On $\mathcal U_D$, the map $F_D$ is uniformly Lipschitz from $H^1$ to $H^1$.
\end{lemma}

\begin{proof}
The derivatives of $a$ are
\[
 Da(z)[h]=-\frac{\langle z,h\rangle_2}{a(z)},\qquad
 D^2a(z)[h,h]=-\frac{\|h\|_2^2}{a(z)}
 -\frac{\langle z,h\rangle_2^2}{a(z)^3}.
\]
Since $a$ is bounded below and $\|\varphi\|_{H^1}^2=1+\lambda$, these give
\begin{equation}\label{gn:eq:sphere-derivatives}
 \|DF_D(z)[h]\|_{H^1}\le K\|h\|_{H^1},\qquad
 \|D^2F_D(z)[h,h]\|_{H^1}\le K\|h\|_{H^1}^2.
\end{equation}
The first bound gives Lipschitz continuity. Since $2<q<2^*$, the nonlinear integral is $C^2$ on $H^1(N)$. Writing $F=F_D(z)$,
\begin{align*}
 DJ_D(z)[h]&=q\int_N|F|^{q-2}F\,DF_D(z)[h],\\
 D^2J_D(z)[h,h]
 &=q(q-1)\int_N|F|^{q-2}\bigl(DF_D(z)[h]\bigr)^2
 +q\int_N|F|^{q-2}F\,D^2F_D(z)[h,h].
\end{align*}
Since $\|F\|_2=1$, $\|F\|_{H^1}^2\le1+3\lambda/2$, and $J_D(z)\ge S$ by H\"older's inequality, the Sobolev embedding on $N$ gives
\begin{equation}\label{gn:eq:nonlinear-log-Hessian}
 |D^2\log J_D(z)[h,h]|\le K\|h\|_{H^1}^2.
\end{equation}
This holds also for sign-changing $F$.

Put $t=B_D(z,z)$. The Cauchy--Schwarz inequality for the positive form $B_D$ gives
\begin{align*}
 D^2\log\bigl(\lambda+B_D(z,z)\bigr)[h,h]
 &=\frac{2B_D(h,h)}{\lambda+t}
 -\frac{4B_D(z,h)^2}{(\lambda+t)^2}\\
 &\ge\frac{2(\lambda-t)}{(\lambda+t)^2}B_D(h,h)
 \ge\frac{4}{9\lambda}B_D(h,h),
\end{align*}
since $0\le t<\lambda/2$. Together with \eqref{gn:eq:nonlinear-log-Hessian} and \eqref{gn:eq:complement-gap}, this gives \eqref{gn:eq:log-Hessian} for sufficiently small $\lambda$.
\end{proof}

\begin{proof}[Uniqueness and stability in Theorem~\ref{gn:thm:small-holes}]
By \eqref{gn:eq:near-maximizer-energy}, choose $\delta_0$ and then $\lambda_0$ such that every admissible $f$ with $\delta_D(f)\le\delta_0$ is represented by $f=F_D(z)$ with $B_D(z,z)<\lambda/4$. In particular, this applies to every nonnegative normalized maximizer. Taking absolute values preserves the quotient on real functions, so each such maximizer also maximizes the quotient over all real functions. Its coordinate is therefore an interior critical point of $\mathcal L_D$. Existence follows from \eqref{gn:eq:uniform-spectral-lower} and Theorem~\ref{gn:thm:threshold}. Strict concavity gives a unique such point $z_D$; set $u_D=F_D(z_D)$.

For $f=F_D(z)$ with $\delta=\delta_D(f)\le\delta_0$, integrate \eqref{gn:eq:log-Hessian} on the segment from $z_D$ to $z$, using $D\mathcal L_D(z_D)=0$. This gives
\[
 \frac{c}{2\lambda}\|z-z_D\|_{H^1}^2
 \le\mathcal L_D(z_D)-\mathcal L_D(z)
 =-\log(1-\delta)\le2\delta.
\]
The Lipschitz estimate for $F_D$ now proves \eqref{gn:eq:deficit-stability}.
\end{proof}

\subsection{Approximation by the first eigenfunction}

\begin{lemma}[First variation at the eigenfunction]\label{gn:lem:eigenfunction-residual}
For small $\lambda$,
\begin{equation}\label{gn:eq:eigenfunction-residual}
 |D\mathcal L_D(0)[h]|\le K\lambda^{1/2}\|h\|_{H^1(N)}
 \qquad(h\in X_D).
\end{equation}
\end{lemma}

\begin{proof}
The first derivative of $\lambda+B_D(z,z)$ vanishes at zero and $DF_D(0)[h]=h$. Let $c_0=V^{-1/2}$. Orthogonality to $\varphi$ gives
\[
 D\mathcal L_D(0)[h]
 =\frac q{\int_N\varphi^q}\int_N
 \bigl(\varphi^{q-1}-c_0^{q-2}\varphi\bigr)h.
\]
The pointwise bound $|t^{q-1}-c_0^{q-2}t|\le K(1+t^{q-2})|t-c_0|$ for $t\ge0$ and H\"older's inequality, with $q'=q/(q-1)$, give
\[
 \|\varphi^{q-1}-c_0^{q-2}\varphi\|_{q'}
 \le K(1+\|\varphi\|_q^{q-2})\|\varphi-c_0\|_q.
\]
Lemma~\ref{gn:lem:small-energy-norm}, applied to $\varphi$, gives $\|\varphi-c_0\|_{H^1}\le K\lambda^{1/2}$. The Sobolev embedding and $\int_N\varphi^q\ge S$ prove \eqref{gn:eq:eigenfunction-residual}.
\end{proof}

\begin{proof}[Spectral estimates in Theorem~\ref{gn:thm:small-holes}]
Equation~\eqref{gn:eq:leading-constant-proof} proves \eqref{gn:eq:spectral-equivalent}. By \eqref{gn:eq:near-maximizer-energy} with zero deficit and Lemma~\ref{gn:lem:small-energy-norm}, $\|u_D-V^{-1/2}\|_{H^1}=O(\lambda^{1/2})$, proving \eqref{gn:eq:constant-approximation}.

Using $D\mathcal L_D(z_D)=0$, integrate the Hessian on the segment from $0$ to $z_D$. Lemmas~\ref{gn:lem:strict-concavity} and~\ref{gn:lem:eigenfunction-residual} give
\[
 \frac c\lambda\|z_D\|_{H^1}^2
 \le D\mathcal L_D(0)[z_D]
 \le K\lambda^{1/2}\|z_D\|_{H^1}.
\]
Thus $\|z_D\|_{H^1}=O(\lambda^{3/2})$. Since $a(z_D)=1+O(\|z_D\|_2^2)=1+O(\lambda^3)$ and $u_D=a(z_D)\varphi+z_D$, this proves \eqref{gn:eq:eigenfunction-approximation}. The exact identity \eqref{gn:eq:exact-sphere-energy} gives
\[
 0\le\|\nabla u_D\|_2^2-\lambda
 =B_D(z_D,z_D)\le\|z_D\|_{H^1}^2=O(\lambda^3),
\]
which is \eqref{gn:eq:energy-asymptotic}. Finally, concavity and maximality imply
\[
 0\le\log\frac{C_D}{W_{M_D}(\varphi_D)}
 =\mathcal L_D(z_D)-\mathcal L_D(0)
 \le D\mathcal L_D(0)[z_D]=O(\lambda^2).
\]
Exponentiation proves \eqref{gn:eq:eigenfunction-quotient}.
\end{proof}

\begin{remark}\label{gn:rem:vanishing-hole-volume}
Since $\widetilde\varphi_D=0$ on $D$, Lemma~\ref{gn:lem:small-energy-norm} gives $V^{-1}\operatorname{Vol}(D)\le\|\widetilde\varphi_D-V^{-1/2}\|_2^2=O(\lambda_D)$. Thus vanishing hole volume follows from the spectral assumption.
\end{remark}

\subsection{Capacity and the first eigenvalue}

The admissible class in \eqref{gn:eq:capacity-definition-intro} is nonempty and weakly closed. Poincar\'e gives a unique minimizer of positive energy. Smoothness of $\partial D$ ensures that $1-v\in H^1_0(M_D)$ for every admissible $v$.

Courtois \cite[Theorem~1.3]{Courtois1995} proved \eqref{gn:eq:capacity-spectral-equivalent} for this capacity. We use the following two-sided estimate.

\begin{proposition}[Spectral-capacity comparison]\label{gn:prop:capacity-spectral}
Let $c_D=\operatorname{cap}_N(D)$. Then
\begin{equation}\label{gn:eq:capacity-spectral-bounds}
 \frac{c_D}{V+c_D/\Lambda}\le\lambda_D\le\frac{c_D}{V}.
\end{equation}
Consequently $c_D\to0$ if and only if $\lambda_D\to0$, and in this limit
\begin{equation}\label{gn:eq:capacity-spectral-equivalent}
 \lambda_D=\frac{c_D}{V}\bigl(1+O(c_D)\bigr).
\end{equation}
\end{proposition}

\begin{proof}
For the capacitary minimizer $v_D$, the Dirichlet test function $1-v_D$ has squared norm $V+\|v_D\|_2^2$ and energy $c_D$, giving $\lambda_D\le c_D/V$.

For the other direction, let $a=V^{-1}\int_N\varphi_D>0$. The function $v=1-\varphi_D/a$ is admissible for the capacity and $c_Da^2\le\lambda_D$. By \eqref{gn:eq:ambient-Poincare},
\[
 Va^2=1-\|\varphi_D-a\|_2^2\ge1-\lambda_D/\Lambda.
\]
Combining the last two inequalities and rearranging gives the lower bound in \eqref{gn:eq:capacity-spectral-bounds}. The equivalence and \eqref{gn:eq:capacity-spectral-equivalent} follow.
\end{proof}

Substituting \eqref{gn:eq:capacity-spectral-equivalent} into \eqref{gn:eq:spectral-equivalent}, with $1-q/2+\beta/2=1-\alpha/2$, proves \eqref{gn:eq:capacity-equivalent} and completes Theorem~\ref{gn:thm:small-holes}.

\begin{remark}[Torsional representation of the capacity]\label{gn:rem:torsion}
If $-\Delta_gt_D=1$ in $M_D$ and $t_D\in H^1_0(M_D)$, then $\operatorname{cap}_N(D)=V^2/\int_{M_D}t_D$. Indeed, $w=1-v$ has $\int w=V$ and $V=\int\nabla w\cdot\nabla t_D$. Cauchy--Schwarz gives the lower bound, with equality at $w=Vt_D/\int t_D$.
\end{remark}

\subsection{Geodesic balls}

Write $\omega_{n-1}=|\mathbb S^{n-1}|$ and $\kappa_n=(n-2)\omega_{n-1}$ for $n\ge3$.

\begin{lemma}[Capacity of a small ball]\label{gn:lem:ball-capacity}
For a fixed $x\in N$,
\begin{equation}\label{gn:eq:ball-capacity}
 \operatorname{cap}_N(B_g(x,r))\sim
 \begin{cases}
 \kappa_n r^{n-2},&n\ge3,\\
 2\pi/|\log r|,&n=2,
 \end{cases}
 \qquad r\downarrow0.
\end{equation}
\end{lemma}

\begin{proof}
Fix $0<\delta$ below the normal-coordinate radius at $x$. Let $f_r$ equal one on $B_r$, zero outside $B_\delta$, and on the intervening annulus set
\[
 f_r(s)=\begin{cases}
 \displaystyle\frac{s^{2-n}-\delta^{2-n}}{r^{2-n}-\delta^{2-n}},&n\ge3,\\[2mm]
 \displaystyle\frac{\log(\delta/s)}{\log(\delta/r)},&n=2.
 \end{cases}
\]
The radial energy in the Riemannian metric differs by a factor $1+O(\delta^2)$ from its Euclidean value
\[
 \frac{\omega_{n-1}}{\displaystyle\int_r^\delta s^{1-n}\,ds}.
\]
Since $0\le f_r\le1$ and $f_r\to0$ almost everywhere, $a_r=V^{-1}\int_Nf_r\to0$. The function $(f_r-a_r)/(1-a_r)$ is admissible for the mean-zero capacity. This gives the required upper bound, up to a factor tending to $1+O(\delta^2)$.

Let $v_r$ be the mean-zero capacitary minimizer. The upper bound and \eqref{gn:eq:ambient-Poincare} imply $v_r\to0$ in $H^1(N)$. Hence its trace on the fixed sphere of radius $\delta$ tends to zero in $L^2$. In Euclidean polar coordinates inside the normal chart, let $b_r(s)$ be the spherical average of $v_r$. Its trace at $s=r$ is one, while $b_r(\delta)\to0$. The Gauss lemma, the volume-density comparison, and Cauchy--Schwarz give
\begin{align*}
 \operatorname{cap}_N(B_g(x,r))
 &\ge(1-C\delta^2)\omega_{n-1}\int_r^\delta s^{n-1}|b_r'(s)|^2\,ds\\
 &\ge(1-C\delta^2)\frac{\omega_{n-1}|1-b_r(\delta)|^2}
 {\displaystyle\int_r^\delta s^{1-n}\,ds}.
\end{align*}
First let $r\downarrow0$ and then $\delta\downarrow0$. Evaluating the denominator proves \eqref{gn:eq:ball-capacity} in both dimension ranges.
\end{proof}

\begin{corollary}[Complements of small geodesic balls]\label{gn:cor:ball-law}
As $r\downarrow0$,
\begin{equation}\label{gn:eq:ball-hole-law}
 C_{N\setminus\overline{B_g(x,r)}}\sim
 \begin{cases}
 V^{1-\alpha/2}\kappa_n^{-\beta/2}r^{-(n-2)\beta/2},&n\ge3,\\
 (2\pi)^{-\beta/2}|\log r|^{\beta/2},&n=2.
 \end{cases}
\end{equation}
\end{corollary}

\begin{proof}
Combine Theorem~\ref{gn:thm:small-holes} and Lemma~\ref{gn:lem:ball-capacity}. In dimension two, $\alpha=2$, so the volume factor is one.
\end{proof}

\begin{remark}[Examples]\label{gn:rem:examples}
If $N$ is a flat torus and $M_r=N\setminus\overline{B_g(x,r)}$, then $C_{M_r}\to\infty$ as $r\downarrow0$, and the nonnegative normalized extremal is unique for sufficiently small $r$. In particular, attainment does not imply positive scalar curvature. Bounded Euclidean domains, on the other hand, have $C_M=C_*$ and no extremal.
\end{remark}

\section{The \texorpdfstring{$L^2$}{L2}-critical nonlinear Schr\"odinger energy}
\label{gn:sec:NLS}

Assume in this section that $q=2+4/n$. Then $\beta=2$, $\alpha=4/n$, and the energy and masses are defined by \eqref{gn:eq:NLS-energy} and \eqref{gn:eq:two-masses}. The same infimum is obtained for real or complex functions: passing from $u$ to $|u|$ preserves mass and the nonlinear integral and cannot increase the gradient energy.

\begin{proof}[Proof of Theorem~\ref{gn:thm:NLS}]
For a function of mass $m$, the definition of $W_M$ gives the exact identity
\begin{equation}\label{gn:eq:energy-quotient}
 \mathcal E_M(u)=\left(\frac12-\frac{m^{2/n}}qW_M(u)\right)\|\nabla u\|_2^2.
\end{equation}
The energy is therefore nonnegative for $m\le m_0(M)$. For $m>m_0(M)$, choose $v$ with $W_M(v)>q/(2m^{2/n})$ and multiply it by a constant to obtain mass $m$. The resulting function has negative energy.

Suppose $0<m<m_*$. Choose $\eta>0$ such that $(C_*+\eta)m^{2/n}<q/2$. The almost-sharp inequality gives, under the constraint $\|u\|_2^2=m$,
\begin{equation}\label{gn:eq:NLS-coercive}
 \mathcal E_M(u)\ge
 \left(\frac12-\frac{C_*+\eta}{q}m^{2/n}\right)\|\nabla u\|_2^2
 -\frac{B_\eta}{q}m^{q/2}.
\end{equation}
A minimizing sequence is therefore bounded in $H^1_0(M)$. Compactness into $L^2$ and $L^q$ preserves the mass and the nonlinear integral, while weak lower semicontinuity of the gradient term gives a minimizer.

For $m>m_*$, a compactly supported approximation to a suitable constant multiple of a Euclidean optimizer gives a function $f$ of mass $m$ with negative Euclidean energy. Apply $P_{x,\varepsilon}$ at an interior point $x$ and multiply by $1+O(\varepsilon^2)$ so that the resulting function has mass $m$. Local comparison gives
\[
 \mathcal E_M(P_{x,\varepsilon}f)
 =\varepsilon^{-2}\left(\frac12\|\nabla f\|_2^2-\frac1q\|f\|_q^q\right)+O(1),
\]
with the same leading term after normalization. Hence $I_M(m)=-\infty$.

For $C_M>C_*$, the interval $(m_0(M),m_*)$ has negative-energy minimizers by the preceding arguments. At $\beta=2$, \eqref{gn:eq:spectral-equivalent} reads $C_D=V^{-2/n}\lambda_D^{-1}(1+O(\lambda_D))$; substitution in \eqref{gn:eq:two-masses} gives \eqref{gn:eq:small-hole-mass}.
\end{proof}

\begin{corollary}[Zero and negative energy standing waves]\label{gn:cor:standing-waves}
If $C_M>C_*$, then $I_M(m_0(M))=0$ is attained by a suitably normalized GN extremal. For $m_0(M)\le m<m_*$, every nonnegative constrained minimizer is positive in $M^\circ$ and solves
\begin{equation}\label{gn:eq:standing-wave}
 -\Delta_g u+\omega u=u^{q-1},\qquad u|_{\partial M}=0,
\end{equation}
for some $\omega>0$. It gives the standing wave $\psi(t,z)=e^{i\omega t}u(z)$ for the focusing Dirichlet equation $i\partial_t\psi+\Delta_g\psi+|\psi|^{q-2}\psi=0$.
\end{corollary}

\begin{proof}
At mass $m_0(M)$, the energy is nonnegative by \eqref{gn:eq:energy-quotient}, and a suitably normalized GN extremal makes it zero. For a constrained minimizer, the multiplier in \eqref{gn:eq:standing-wave} satisfies $\omega m=\|u\|_q^q-\|\nabla u\|_2^2$. Since $\|u\|_q^q\ge(q/2)\|\nabla u\|_2^2$ at nonpositive energy, $\omega>0$. Elliptic regularity and the maximum principle give positivity; substitution verifies the nonlinear Schr\"odinger equation.
\end{proof}

\begin{corollary}[Capacity and the threshold for negative energy]\label{gn:cor:capacity-mass}
For a family $M_D=N\setminus\overline D$ with connected complements and $\lambda_D\to0$,
\begin{equation}\label{gn:eq:capacity-mass}
 m_0(M_D)\sim\left(\frac q2\right)^{n/2}V^{1-n/2}
 \operatorname{cap}_N(D)^{n/2}.
\end{equation}
Consequently, for every fixed $m\in(0,m_*)$, all sufficiently small holes admit a negative-energy constrained minimizer of mass $m$.
\end{corollary}

\begin{proof}
Use \eqref{gn:eq:small-hole-mass} and \eqref{gn:eq:capacity-spectral-equivalent}. The last conclusion follows from $m_0(M_D)\to0$ and Theorem~\ref{gn:thm:NLS}.
\end{proof}

\section{Further applications}
\label{gn:sec:applications}

The sharp constants and stability estimates above also provide input for nonlinear evolution equations. We indicate two applications and the steps needed to develop them. We do not undertake the corresponding analysis of the evolution problems here.

\subsection{Fast diffusion and extinction estimates}

Consider the fast diffusion equation $\partial_t\rho=\Delta_g(\rho^m)$ with $\rho\ge0$ and $\rho^m\in H^1_0(M)$, where $0<m<1$. Choose $q=1+1/m$ and assume $q<2^*$; thus $m>(n-2)/(n+2)$ when $n\ge3$. In this subsection, $C_M$, $\alpha$, and $\beta$ refer to this exponent $q$. The natural starting point is the dissipation identity
\[
 Y(t):=\int_M\rho^{m+1}\,dV_g,\qquad
 Y'(t)=-(m+1)\int_M|\nabla(\rho^m)|^2\,dV_g.
\]
Applying the GN inequality to $f=\rho^m$ and then using $\|f\|_2\le\operatorname{Vol}(M)^{1/2-1/q}\|f\|_q$ gives
\begin{equation}\label{gn:eq:diffusion-dissipation}
 Y'(t)\le-(m+1)K_M Y(t)^{2m/(m+1)},\qquad
 K_M=C_M^{-2/\beta}\operatorname{Vol}(M)^{-2\alpha/(nq)}.
\end{equation}
For solutions satisfying the dissipation identity, integration yields the extinction-time upper bound $T_{\mathrm{ext}}\le Y(0)^{(1-m)/(m+1)}/((1-m)K_M)$. A complementary lower bound follows by testing the equation against the positive first Dirichlet eigenfunction $\Phi_1$, normalized by $\int_M\Phi_1\,dV_g=1$. Concavity of $s\mapsto s^m$ gives
\[
 \frac{d}{dt}\int_M\rho\Phi_1\,dV_g
 \ge-\lambda_1(M)\left(\int_M\rho\Phi_1\,dV_g\right)^m,
 \qquad
 T_{\mathrm{ext}}\ge
 \frac{\left(\int_M\rho(0)\Phi_1\,dV_g\right)^{1-m}}
 {(1-m)\lambda_1(M)}.
\]

For $M_D=N\setminus\overline D$ as in Theorem~\ref{gn:thm:small-holes}, the sharp-constant asymptotic and $\operatorname{Vol}(M_D)\to V$ give
\begin{equation}\label{gn:eq:diffusion-capacity}
 K_{M_D}\sim\lambda_D V^{(1-m)/(m+1)},\qquad
 \lambda_D^{-1}\sim\frac{V}{\operatorname{cap}_N(D)}.
\end{equation}
For initial data obtained by restricting a fixed bounded, nonnegative, nonzero function on $N$, the normalized eigenfunction converges to $V^{-1}$, so the lower bound has a positive limiting numerator. The two estimates therefore give the same order $\lambda_D^{-1}$ for the extinction time. This supplies a direct route from the GN asymptotic to quantitative dependence of extinction on hole capacity. A treatment of weak solutions, rescaled extinction profiles, and convergence rates would require further analysis. The use of interpolation and stability inequalities in entropy methods is developed in \cite{CarlenFigalli13,BDNSMemoir}.

\subsection{Stability of standing waves}

Return to $q=2+4/n$. Section~\ref{gn:sec:NLS} establishes the variational mass thresholds and existence of standing waves. These results also provide a starting point for studying dynamical stability. At the zero-energy mass on $M_D$, put $U_D=\sqrt{m_0(M_D)}u_D$. For $U=\sqrt{m_0(M_D)}f$ with $f\ge0$ and $\|f\|_2=1$, \eqref{gn:eq:energy-quotient} gives
\[
 \mathcal E_{M_D}(U)
 =\frac{m_0(M_D)}2\|\nabla f\|_2^2\delta_D(f).
\]
Together with $\|\nabla f\|_2^2\ge\lambda_D$, the deficit estimate in Theorem~\ref{gn:thm:small-holes} bounds $\|U-U_D\|_{H^1}^2$ by a constant times $\mathcal E_{M_D}(U)$ when the relative deficit is small. This gives quantitative control of nonnegative near-minimizers. To study orbital stability, one would combine compactness of constrained minimizing sequences with conservation of mass and energy for a well-posed Dirichlet flow, and account for the complex phase. The capacity law for $m_0(M_D)$ and the approximation of $u_D$ by $\varphi_D$ make the small-hole limit a natural setting for this analysis. We leave these dynamical questions for future work.

\subsection*{Acknowledgements}
The authors thank IIT Bombay for providing ideal working conditions.

\bibliographystyle{alpha}
\bibliography{GN}

\end{document}